\documentclass{amsart}
\usepackage{amsfonts}

\usepackage{amscd}
\usepackage{thmdefs}

\theoremstyle{definition}
\theoremstyle{remark}
\numberwithin{equation}{section}

\begin{document}
\title{On tensor products of semistable lattices}
\author{Ehud de Shalit and Ori Parzanchevski}
\address{Institute of mathematics, Hebrew University, Jerusalem 91904 ISRAEL}
\email{deshalit@math.huji.ac.il, parzan@math.huji.ac.il}
\date{November 15, 2006}
\maketitle

\begin{abstract}
We give an elementary proof of low rank cases of the conjecture that the
tensor product of two semistable Euclidean lattices is again
semistable.\bigskip 
\end{abstract}

The degree of a lattice $L$ in Euclidean space $\Bbb{R}^{n}$ is defined by
the expression 
\begin{equation}
\deg (L)=-\log vol(\Bbb{R}^{n}/L),
\end{equation}
where $vol$ stands for volume. Its rank\emph{\ }is $rk(L)=n,$ and the \emph{%
slope} $\mu (L)$ is the ratio 
\begin{equation}
\mu (L)=\frac{\deg (L)}{rk(L)}.
\end{equation}
The lattice $L$ is called \emph{semistable} if for any (lower rank)
sublattice $S,$ $\mu (S)\le \mu (L).$ This notion is invariant under
scaling, so if we normalize $L$ to be unimodular ($\mu (L)=0$), it means
that the volume of a fundamental parallelopiped of any sublattice $S$ is at
least 1. Thus the shortest nonzero vector in $L$ has norm at least 1, every
two independent vectors span a parallelogram of area at least 1, and so on.

The following problem seems to be open. The first author learned about it
more than ten years ago from J.-B. Bost.\medskip

\textbf{Question. }Is $L\otimes M$ semistable whenever $L$ and $M$
are?\medskip

The notion of semistability and the question can be easily generalized to
``metrized vector bundles over $Spec(\mathcal{O}_{K})"$ for any number field 
$K.$ They come up naturally in Arakelov geometry. However, it is already
interesting, and probably not less difficult, over $Spec(\Bbb{Z)},$ where it
amounts to the statement made above. It is modeled on similar results for
vector bundles over curves, and for filtered vector spaces (see [Fa],[To]).
The corresponding result for filtered vector spaces is surprisingly
difficult, although it is only a statement in linear algebra, and this leads
one to suspect that the question for lattices is not easy either.\medskip

The purpose of this note is to give an affirmative answer in some special
cases.

\begin{theorem}
Let $L$ and $M$ be two semistable lattices. If $S$ is a sublattice of $%
L\otimes M$ and $rk(S)\le 3$ then 
\begin{equation}
\mu (S)\le \mu (L\otimes M).
\end{equation}
\end{theorem}

Bost had obtained other special cases using Geometric Invariant Theory
(unpublished). We should also stress that our elementary methods do not
generalize to ``metrized vector bundles over $Spec(\mathcal{O}_{K})"$ when $%
K $ is not $\Bbb{Q}.$

\section{The proof}

\subsection{Preliminaries on Euclidean spaces}

If $V$ and $W$ are Euclidean spaces (finite dimensional inner product spaces
over $\Bbb{R})$ then so is $V\otimes W$ with the inner product 
\begin{equation}
(v\otimes w,v^{\prime }\otimes w^{\prime })_{V\otimes W}=(v,v^{\prime
})_{V}(w,w^{\prime })_{W}.
\end{equation}
If $\left\{ e_{i}\right\} $ is an orthonormal basis of $V$ and $\left\{
e_{j}^{\prime }\right\} $ an orthonormal basis of $W$ then $\left\{
e_{i}\otimes e_{j}^{\prime }\right\} $ is an orthonormal basis of $V\otimes
W.$ In this way $V^{\otimes k}$ is endowed with a Euclidean structure.

The $k$-th exterior power $\bigwedge^{k}V$ is the quotient of $V^{\otimes k}$
by the subspace $\mathcal{N}$ spanned by tensors $v_{1}\otimes \cdot \cdot
\cdot \otimes v_{k}$ in which, for some $i\neq j$, $v_{i}=v_{j}.$ The
orthogonal complement $\mathcal{N}^{\perp }$ is the space of \emph{%
alternating} $k$\emph{-tensors}, spanned by 
\begin{equation}
\frac{1}{k!}\sum_{\sigma \in S_{k}}sgn(\sigma )v_{\sigma (1)}\otimes \cdot
\cdot \cdot \otimes v_{\sigma (k)},
\end{equation}
and this tensor projects modulo $\mathcal{N}$ to $v_{1}\wedge \cdot \cdot
\cdot \wedge v_{k},$ the image modulo $\mathcal{N}$ of $v_{1}\otimes \cdot
\cdot \cdot \otimes v_{k}.$ We may therefore identify $\bigwedge^{k}V$ with $%
\mathcal{N}^{\perp }$. As such it inherits from $V^{\otimes k}$ a Euclidean
structure. If $\left\{ e_{i}\right\} $ is an orthonormal basis of $V$ then $%
e_{I}=e_{i_{1}}\wedge \cdot \cdot \cdot \wedge e_{i_{k}},$ for $%
I=(i_{1}<\dots <i_{k}),$ is a basis of $\bigwedge^{k}V,$ which is orthogonal
but not normalized: the norm of $e_{I}$ is $1/\sqrt{k!}.$ To correct it, we
modify the inner product that $\bigwedge^{k}V$ inherits from $V^{\otimes k}$
by a factor of $k!$ and set 
\begin{equation}
(v_{1}\wedge \cdot \cdot \cdot \wedge v_{k},v_{1}^{\prime }\wedge \cdot
\cdot \cdot \wedge v_{k}^{\prime })=\sum_{\sigma \in S_{k}}sgn(\sigma
)\prod_{i=1}^{k}(v_{i},v_{\sigma (i)}^{\prime }).
\end{equation}
In this inner product the $e_{I}$ form an orthonormal basis. Moreover, the 
\emph{volume} of the $k$-dimensional parallelopiped spanned by $v_{1},\dots
,v_{k}$ in $V$ is nothing but the \emph{norm} $|v_{1}\wedge \cdot \cdot
\cdot \wedge v_{k}|.$ To see it, choose an orthonormal basis $e_{i}$ of $V$
the first $k$ vectors of which span $Span\left\{ v_{1},\dots ,v_{k}\right\} $
(assuming the $v_{i}$ are linearly independent) and observe that 
\begin{equation}
v_{1}\wedge \cdot \cdot \cdot \wedge v_{k}=\det (a_{ij})e_{1}\wedge \cdot
\cdot \cdot \wedge e_{k}
\end{equation}
where $v_{i}=\sum a_{ij}e_{j}.$

\subsection{Length and norm in $L\otimes M,$ and the rank 1 case of the
theorem}

To begin the proof, note that 
\begin{equation}
\mu (L\otimes M)=\mu (L)+\mu (M).
\end{equation}
If $\alpha $ is a vector in $L\otimes M$ we denote by $l(\alpha )$ the \emph{%
length} of $\alpha ,$ which is the minimal number $l$ such that 
\begin{equation}
\alpha =\sum_{k=1}^{l}u_{k}\otimes u_{k}^{\prime }.
\end{equation}
(To avoid abuse of language, $|\alpha |$ will always be called the norm of $%
\alpha $, and not its length.) In such a case, the $u_{k}$ and the $%
u_{k}^{\prime }$ are linearly independent vectors in $L$ and $M$
respectively. Thus the length of a vector is at most $\min (rk(L),rk(M)).$

Let us write, for any lattice $M$%
\begin{equation}
\mu _{\max }(M)=\max \left\{ \mu (L);\,L\subseteq M\right\} .
\end{equation}
The maximum is over all sublattices of all ranks (it is easily seen that the
maximum is attained). A unimodular lattice $M$ is semistable if and only if $%
\mu _{\max }(M)=0.$

\begin{proposition}
Let $L$ and $M$ be any two lattices, normalized so that $\mu _{\max }(L)=\mu
_{\max }(M)=0$ and $\alpha \in L\otimes M.$ Then 
\begin{equation}
|\alpha |\ge \sqrt{l(\alpha )}.
\end{equation}
\end{proposition}

\begin{proof}
Writing $\alpha =\sum_{k=1}^{l}u_{k}\otimes u_{k}^{\prime }$, a direct
computation gives 
\begin{equation}
|\alpha |^{2}=\sum_{k,m}(u_{k},u_{m})(u_{k}^{\prime },u_{m}^{\prime }).
\end{equation}
\newline
Consider the \emph{Gramians }$A=\left( (u_{k},u_{m})\right) $ and $A^{\prime
}=\left( (u_{k}^{\prime },u_{m}^{\prime })\right) .$ These are symmetric
positive definite $l$ by $l$ matrices with $\det (A)\ge 1$ and $\det
(A^{\prime })\ge 1$ because we normalized $L$ and $M$ so that the volume of
the parallelopiped formed by the $u_{k}$ or the $u_{k}^{\prime }$ is at
least 1.

We are indebted to Assaf Goldberger for pointing out the following fact.

\begin{lemma}
Let $A$ and $A^{\prime }$ be two symmetric positive definite $l$ by $l$
matrices. Then the spectrum of $AA^{\prime }$ is real and positive.
\end{lemma}

The matrix $AA^{\prime },$ although not necessarily symmetric, is conjugate
to 
\begin{equation}
A^{1/2}A^{\prime }A^{1/2}=(A^{1/2})^{t}A^{\prime }A^{1/2},
\end{equation}
which is symmetric positive definite since $A^{\prime }$ is.

We now continue the proof of the proposition. By our assumption, the product
of the eigenvalues of $AA^{\prime }$ is at least 1, and the lemma implies
that they are real and positive. By the inequality between the geometric and
arithmetic means, $Tr(AA^{\prime })\ge l.$ However, 
\begin{equation}
Tr(AA^{\prime })=|\alpha |^{2},
\end{equation}
so $|\alpha |\ge \sqrt{l}$ as needed.
\end{proof}

The proposition proves the case $\mathbf{rk(S)=1}$ of the theorem, but to
prove the cases $rk(S)=2$ or $3$ we shall need its full strength.

\subsection{The rank 2 case of the theorem}

Let now $S$ be a sublattice of \textbf{rank }$2$ in $L\otimes M.$ Let $%
\alpha $ be a nonzero vector of shortest norm in $S,$ and $\beta $ another
vector in $S$ such that $S=\Bbb{Z}\alpha +\Bbb{Z}\beta $. Subtracting a
multiple of $\alpha $ from $\beta $ we may assume that the angle between $%
\alpha $ and $\beta $ is between $60^{\circ }$ and $120^{\circ }.$ It
follows that 
\begin{equation}
|\alpha \wedge \beta |\ge \frac{\sqrt{3}}{2}|\alpha ||\beta |.
\end{equation}
If either of $\alpha $ or $\beta $ is indecomposable (of length $\ge 2$) we
are done, since its norm is then at least $\sqrt{2}$ and the other's norm is
at least 1, but $\sqrt{2}\sqrt{3}>2.$ We may therefore assume that 
\begin{equation}
\alpha =v_{1}\otimes w_{1},\,\,\beta =v_{2}\otimes w_{2}
\end{equation}
are both decomposable tensors. We now use the identity 
\begin{eqnarray}
|(v_{1}\otimes w_{1})\wedge (v_{2}\otimes w_{2})|^{2}
&=&|v_{1}|^{2}|v_{2}|^{2}|w_{1}\wedge w_{2}|^{2}+|v_{1}\wedge
v_{2}|^{2}|w_{1}|^{2}|w_{2}|^{2}  \notag \\
&&-|v_{1}\wedge v_{2}|^{2}|w_{1}\wedge w_{2}|^{2},
\end{eqnarray}
which follows at once from the fact that 
\begin{equation}
|\alpha \wedge \beta |^{2}=|\alpha |^{2}|\beta |^{2}-(\alpha ,\beta )^{2}.
\end{equation}
If $v_{1}$ and $v_{2}$ are proportional, then $w_{1}$ and $w_{2}$ are
independent, and $|v_{1}|^{2}|v_{2}|^{2}|w_{1}\wedge w_{2}|^{2}$ $\ge 1,$
while the second and the third terms vanish. Otherwise, the first term is
larger than the third, and the second is $\ge 1$ by our assumption. This
concludes the rank 2 case.

\subsection{The rank 3 case of the theorem}

The \textbf{rank 3} case is handled similarly, but the details are more
complicated. Let $rk(S)=3.$ Let $\alpha $ be a shortest nonzero vector in $S$
and $\beta $ a second shortest vector independent from $\alpha ,$ and assume
as before that the angle between $\alpha $ and $\beta $ is between $%
60^{\circ }$ and $120^{\circ }.$ Complete to a basis of $S$ by a vector $%
\gamma ,$ $|\gamma |\ge |\beta |\ge |\alpha |\ge 1.$ Using orthonormal
coordinates $x,y,z$ in the real subspace spanned by $S,$ such that the $x$%
-axis is in the $\alpha $-direction and the $(x,y)$-plane is the plane
spanned by $\alpha $ and $\beta ,$ we may assume, subtracting from $\gamma $
a suitable integral linear combination of $\alpha $ and $\beta $ that 
\begin{equation}
\gamma =(x,y,z)
\end{equation}
with $|x|\le |\alpha |/2$ and $|y|\le |\beta |/2.$ This is because the
rectangle 
\begin{equation}
\left\{ |x|\le |\alpha |/2,\,\,|y|\le |\beta |\sin \theta /2\right\} ,
\end{equation}
where $\theta $ is the angle between $\alpha $ and $\beta ,$ is a
fundamental domain for $\Bbb{Z}\alpha +\Bbb{Z}\beta .$ We now have 
\begin{equation}
z^{2}\ge |\gamma |^{2}-(|\alpha |^{2}+|\beta |^{2})/4\ge |\gamma |^{2}/2
\end{equation}
and 
\begin{equation}
|\alpha \wedge \beta |\ge 1
\end{equation}
(by the rank 2 case), so 
\begin{equation}
|\alpha \wedge \beta \wedge \gamma |^{2}=|\alpha \wedge \beta |^{2}z^{2}\ge
|\gamma |^{2}/2.
\end{equation}
If $l(\alpha )$, $l(\beta )$ or $l(\gamma )\ge 2,$ then by the proposition
and the fact that $|\gamma |\ge |\beta |\ge |\alpha |,$ we must have $%
|\gamma |^{2}\ge 2,$ and we are done.

There remains the case $\alpha =v_{1}\otimes w_{1},$ $\beta =v_{2}\otimes
w_{2}$ and $\gamma =v_{3}\otimes w_{3},$ in which the proposition does not
help us. The key in this case is the following identity.

\begin{lemma}
We have an 18-term relation 
\begin{eqnarray}
|\alpha \wedge \beta \wedge \gamma |^{2}
&=&|v_{1}|^{2}|v_{2}|^{2}|v_{3}|^{2}|w_{1}\wedge w_{2}\wedge
w_{3}|^{2}+|v_{1}\wedge v_{2}\wedge
v_{3}|^{2}|w_{1}|^{2}|w_{2}|^{2}|w_{3}|^{2}  \notag \\
&&-\frac{1}{2}\sum_{i=1}^{3}|v_{i}|^{2}|v_{i^{\prime }}\wedge v_{i^{\prime
\prime }}|^{2}|w_{i}|^{2}|w_{i^{\prime }}\wedge w_{i^{\prime \prime }}|^{2} 
\notag \\
&&+\frac{1}{2}\sum_{i\neq j}|v_{i}|^{2}|v_{i^{\prime }}\wedge v_{i^{\prime
\prime }}|^{2}|w_{j}|^{2}|w_{j^{\prime }}\wedge w_{j^{\prime \prime }}|^{2} 
\notag \\
&&-\frac{1}{2}\sum_{i=1}^{3}|v_{i}|^{2}|v_{i^{\prime }}\wedge v_{i^{\prime
\prime }}|^{2}|w_{1}\wedge w_{2}\wedge w_{3}|^{2}  \notag \\
&&-\frac{1}{2}\sum_{i=1}^{3}|v_{1}\wedge v_{2}\wedge
v_{3}|^{2}|w_{i}|^{2}|w_{i^{\prime }}\wedge w_{i^{\prime \prime }}|^{2} 
\notag \\
&&+\frac{1}{2}|v_{1}\wedge v_{2}\wedge v_{3}|^{2}|w_{1}\wedge w_{2}\wedge
w_{3}|^{2},
\end{eqnarray}
where we have used the convention that $\{i^{\prime },i^{\prime \prime }\}$
are the two indices complementary to $i$.
\end{lemma}

\begin{proof}
\emph{\ }Expanding the grammian of three vectors in Euclidean space we get 
\begin{eqnarray}
|\alpha _{1}\wedge \alpha _{2}\wedge \alpha _{3}|^{2}
&=&\prod_{i=1}^{3}|\alpha _{i}|^{2}-\sum_{i=1}^{3}|\alpha _{i}|^{2}(\alpha
_{i^{\prime }},\alpha _{i^{\prime \prime }})^{2}  \notag \\
&&+2(\alpha _{1},\alpha _{2})(\alpha _{2},\alpha _{3})(\alpha _{3},\alpha
_{1}).
\end{eqnarray}
For our decomposable tensors we get the expression 
\begin{eqnarray}
|\alpha \wedge \beta \wedge \gamma |^{2}
&=&\prod_{i=1}^{3}|v_{i}|^{2}|w_{i}|^{2}-\sum_{i=1}^{3}|v_{i}|^{2}(v_{i^{%
\prime }},v_{i^{\prime \prime }})^{2}|w_{i}|^{2}(w_{i^{\prime
}},w_{i^{\prime \prime }})^{2}  \notag \\
&&+2(v_{1},v_{2})(v_{2},v_{3})(v_{3},v_{1})(w_{1},w_{2})(w_{2},w_{3})(w_{3},w_{1}).
\end{eqnarray}
We proceed to replace all the inner products by expressions involving only
wedge-products (sines instead of cosines). The elements of the form $%
(x,y)^{2}$ can be replaced by $|x|^{2}|y|^{2}-|x\wedge y|^{2}$ as in the
case of rank 2. For $(x,y)(y,z)(z,x)$ use 
\begin{eqnarray}
2(x,y)(y,z)(z,x) &=&|x\wedge y\wedge
z|^{2}-|x|^{2}|y|^{2}|z|^{2}+|x|^{2}(y,z)^{2}+\cdot \cdot \cdot  \notag \\
&=&|x\wedge y\wedge
z|^{2}-|x|^{2}|y|^{2}|z|^{2}+|x|^{2}(|y|^{2}|z|^{2}-|y\wedge z|^{2})+\cdot
\cdot \cdot
\end{eqnarray}
This is an 5-term expression and when we multiply the one for the $%
v_{i}^{\prime }s$ with the one for the $w_{i}^{\prime }s$ we get 25 terms,
and a total of 38 terms for $|\alpha \wedge \beta \wedge \gamma |^{2}.$ A
direct computation (carried out by a computer) reveals that there are many
cancellations, and the 18-term relation falls out.
\end{proof}

We continue the proof of the theorem, when $\alpha $, $\beta $ and $\gamma $
are decomposable as above. Choose an orthonormal basis $%
e_{1},e_{2},e_{3},... $ of $L_{\Bbb{R}}$ so that $v_{1}\in \langle
e_{1}\rangle ,$ $v_{2}\in \langle e_{1},e_{2}\rangle ,$ $v_{3}\in \langle
e_{1},e_{2},e_{3}\rangle $ and similarly an orthonormal basis $e_{1}^{\prime
},e_{2}^{\prime },e_{3}^{\prime },...$ of $M_{\Bbb{R}}$ putting the $w_{i}$
in a triangular form. Assume first that the $v_{i}$ are linearly
independent. Expanding in the orthonormal basis $e_{i}\otimes e_{j}^{\prime }
$ of $L_{\Bbb{R}}\otimes M_{\Bbb{R}}$ and in the corresponding orthonormal
basis of of $\bigwedge^{3}(L_{\Bbb{R}}\otimes M_{\Bbb{R}})$ (recall the
normalization from 1.1) 
\begin{eqnarray}
|\alpha \wedge \beta \wedge \gamma |^{2} &\ge
&\sum_{i=1}^{2}\sum_{j=1}^{3}coef^{2}((e_{1}\otimes e_{1}^{\prime })\wedge
(e_{2}\otimes e_{i}^{\prime })\wedge (e_{3}\otimes e_{j}^{\prime }))  \notag
\\
&=&|v_{1}\wedge v_{2}\wedge v_{3}|^{2}|w_{1}|^{2}|w_{2}|^{2}|w_{3}|^{2}\ge 1.
\end{eqnarray}

Assume next that both the $v_{i}$ and the $w_{i}$ are linearly dependent. If 
$v_{1}$ and $v_{2}$ are proportional then $w_{1}$ and $w_{2}$ can not be
proportional, otherwise $rk(S)\le 2.$ In this case we may assume that $%
v_{1},v_{2}\in \langle e_{1}\rangle $ and $v_{3}\in \langle
e_{1},e_{2}\rangle ,$ while $w_{1}\in \langle e_{1}^{\prime }\rangle $ and $%
w_{2},w_{3}\in \langle e_{1}^{\prime },e_{2}^{\prime }\rangle .$ We find 
\begin{eqnarray}
|\alpha \wedge \beta \wedge \gamma |^{2}
&=&\sum_{i=1}^{2}coef^{2}((e_{1}\otimes e_{1}^{\prime })\wedge (e_{1}\otimes
e_{2}^{\prime })\wedge (e_{2}\otimes e_{i}^{\prime }))  \notag \\
&=&|v_{1}||v_{2}||v_{1}\wedge v_{3}||v_{2}\wedge v_{3}||w_{1}\wedge
w_{2}|^{2}|w_{3}|^{2}\ge 1.
\end{eqnarray}

There remains the case where the $v_{i}$ and the $w_{i}$ are linearly
dependent, but no two vectors in each triplet are proportional.

Introduce the notation 
\begin{equation}
\lambda _{i}=|v_{i}||v_{i^{\prime }}\wedge v_{i^{\prime \prime }}|
\end{equation}
and similarly $\mu _{i}=|w_{i}||w_{i^{\prime }}\wedge w_{i^{\prime \prime
}}|.$ Note that 
\begin{equation*}
\lambda _{i}=|v_{1}||v_{2}||v_{3}||\sin \theta _{i}|
\end{equation*}
where $\theta _{i}$ is the angle between $v_{i^{\prime }}$ and $v_{i^{\prime
\prime }}$ in the plane spanned by the $v_{i}^{\prime }s.$ We may assume
that $\theta _{i}=\theta _{i^{\prime }}+\theta _{i^{\prime \prime }}.$ A
direct computation shows then that 
\begin{equation}
\lambda _{i}^{2}-\lambda _{i^{\prime }}^{2}-\lambda _{i^{\prime \prime
}}^{2}=\pm 2\lambda _{i^{\prime }}\lambda _{i^{\prime \prime }}\cos \theta
_{i}.
\end{equation}
Similarly denote by $\omega _{i}$ the angle between $w_{i^{\prime }}$ and $%
w_{i^{\prime \prime }}.$ We now observe that in the 18-term relation, nine
terms drop out and the remaining nine give 
\begin{eqnarray}
|\alpha \wedge \beta \wedge \gamma |^{2} &=&\frac{1}{2}\left\{ 
\begin{array}{lll}
-\lambda _{1}^{2}\mu _{1}^{2} & +\lambda _{1}^{2}\mu _{2}^{2} & +\lambda
_{1}^{2}\mu _{3}^{2} \\ 
+\lambda _{2}^{2}\mu _{1}^{2} & -\lambda _{2}^{2}\mu _{2}^{2} & +\lambda
_{2}^{2}\mu _{3}^{2} \\ 
+\lambda _{3}^{2}\mu _{1}^{2} & +\lambda _{3}^{2}\mu _{2}^{2} & -\lambda
_{3}^{2}\mu _{3}^{2}
\end{array}
\right\}  \notag \\
&=&\frac{1}{2}\left\{ 
\begin{array}{c}
\lambda _{1}^{2}(-\mu _{1}^{2}+\mu _{2}^{2}+\mu _{3}^{2})+\lambda
_{2}^{2}(\mu _{1}^{2}-\mu _{2}^{2}+\mu _{3}^{2}) \\ 
+(\lambda _{1}^{2}+\lambda _{2}^{2}\pm 2\lambda _{1}\lambda _{2}\cos \theta
_{3})(\mu _{1}^{2}+\mu _{2}^{2}-\mu _{3}^{2})
\end{array}
\right\}  \notag \\
&=&\lambda _{1}^{2}\mu _{2}^{2}+\lambda _{2}^{2}\mu _{1}^{2}\pm \lambda
_{1}\lambda _{2}\cos \theta _{3}(\mu _{1}^{2}+\mu _{2}^{2}-\mu _{3}^{2}) 
\notag \\
&=&\lambda _{1}^{2}\mu _{2}^{2}+\lambda _{2}^{2}\mu _{1}^{2}\pm 2\lambda
_{1}\lambda _{2}\mu _{1}\mu _{2}\cos \theta _{3}\cos \omega _{3}  \notag \\
&\ge &\lambda _{1}^{2}\mu _{2}^{2}+\lambda _{2}^{2}\mu _{1}^{2}-\lambda
_{1}^{2}\mu _{2}^{2}\cos ^{2}\theta _{3}-\lambda _{2}^{2}\mu _{1}^{2}\cos
^{2}\omega _{3}  \notag \\
&=&\lambda _{1}^{2}\mu _{2}^{2}\sin ^{2}\theta _{3}+\lambda _{2}^{2}\mu
_{1}^{2}\sin ^{2}\omega _{3}  \notag \\
&\ge &\frac{\lambda _{1}^{2}\mu _{2}^{2}}{|v_{1}|^{2}|v_{2}|^{2}}+\frac{%
\lambda _{2}^{2}\mu _{1}^{2}}{|w_{1}|^{2}|w_{2}|^{2}}  \notag \\
&=&\frac{|v_{2}\wedge v_{3}|^{2}|w_{2}|^{2}|w_{1}\wedge w_{3}|^{2}}{%
|v_{2}|^{2}}+\frac{|w_{2}\wedge w_{3}|^{2}|v_{2}|^{2}|v_{1}\wedge v_{3}|^{2}%
}{|w_{2}|^{2}}  \notag \\
&\ge &\frac{|w_{2}|^{2}}{|v_{2}|^{2}}+\frac{|v_{2}|^{2}}{|w_{2}|^{2}}\ge 2.
\end{eqnarray}
\medskip This concludes the proof of the theorem.

\subsection{An argument from duality}

Duality allows us to conclude more. Quite generally, if 
\begin{equation}
0\rightarrow P\rightarrow Q\rightarrow R\rightarrow 0
\end{equation}
is an exact sequence of Euclidean lattices (this means that $Q/P$ is
torsion-free and that the metrics on $P_{\Bbb{R}}$ and $R_{\Bbb{R}}$ are the
ones induced from the metric on $Q_{\Bbb{R}}$) then 
\begin{equation}
\mu (Q)=\frac{rkP}{rkQ}\mu (P)+\frac{rkR}{rkQ}\mu (R).
\end{equation}
If $R^{\vee }$ is the dual lattice, $\mu (R^{\vee })=-\mu (R).$ From these
two observations it follows at once that if $L$ is semistable, so is $%
L^{\vee }.$

Let $L$ and $M$ be two unimodular semistable lattices of ranks $n$ and $m$.
Then $L^{\vee }$ and $M^{\vee }$ are unimodular and semistable too. Suppose
we know the desired results for all $S$ of some rank $s.$ Let $P$ be a rank $%
nm-s$ (we shall say it has \emph{corank }$s$) sublattice of $L\otimes M$ for
which we want to prove that $\mu (P)\le 0.$ We may assume that $\,(L\otimes
M)/P$ is torsion free. We equip it with the quotient metric and call it $R$.
Then $R$ is a rank $s$ lattice and $R^{\vee }=P^{\perp }$ is a rank $s$
sublattice of $L^{\vee }\otimes M^{\vee }.$ By what we have seen already, $%
\mu (R^{\vee })\le 0,$ but 
\begin{equation}
\mu (P)=-\frac{s}{nm-s}\mu (R)
\end{equation}
since $\mu (L\otimes M)=0,$ and $\mu (R)=-\mu (R^{\vee }).$ This shows that
if we know the desired result for all rank $s$ lattices, we also know it for
all corank $s$ lattices.\medskip

\begin{corollary}
If $L$ and $M$ are semistable and they are both of rank 2, or one of them is
of rank 2 and the other is of rank 3, then $L\otimes M$ is
semistable.\medskip
\end{corollary}

\begin{equation*}
\text{\textbf{References}}
\end{equation*}

[Fa] Faltings, G.: \emph{Mumford-Stabilit\"{a}t in der algebraischen
Geometrie}, Proceedings of the 1994 ICM.\medskip

[To] Totaro, B.: \emph{Tensor products in }$p$\emph{-adic Hodge theory, }%
Duke Math.J. \textbf{83} (1996), 79-104.

\end{document}